\documentclass[12pt,reqno]{amsart}

\usepackage{amssymb}
\usepackage{amscd}
\usepackage{amsfonts}
\usepackage{mathrsfs}
\usepackage{setspace}
\usepackage{mathpazo}
\usepackage{hyperref}
\usepackage{graphicx}

\newtheorem{theorem}{Theorem}[section]
\newtheorem{lemma}[theorem]{Lemma}
\newtheorem{proposition}[theorem]{Proposition}
\newtheorem{corollary}[theorem]{Corollary}

\theoremstyle{definition}
\newtheorem{definition}[theorem]{Definition}

\theoremstyle{remark}

\renewcommand{\leq}{\leqslant}
\renewcommand{\geq}{\geqslant}

\newcommand\diam{\operatorname{diam}}

\newcommand\Out{\operatorname{Out}}

\newcommand\Aut{\operatorname{Aut}}

\newcommand\prim{\operatorname{prim}}
\newcommand\perm{\operatorname{perm}}
\newcommand\width{\operatorname{width}}
\newcommand\irred{\operatorname{irred}}
\newcommand\tildef{f}

\def\F{\mathbf{F}}
\def\R{\mathbf{R}}
\def\C{\mathbf{C}}
\def\Z{\mathbf{Z}}

\def\N{\mathbf{N}}
\newcommand\U{\operatorname{U}}
\renewcommand\O{\operatorname{O}}

\def\eps{\varepsilon}

\numberwithin{equation}{section}

\begin{document}

\title[Matrix coefficients, transitive sets]{On the width of transitive sets: bounds on matrix coefficients of finite groups}


\author{Ben Green}
\thanks{The author is supported by a Simons Investigator Award and is grateful to the Simons Foundation for their support. He also wishes to thank Ashwin Sah, Mehtaab Sawhney and Yufei Zhao for pointing out two significant errors in the published version of the paper, and for providing corrections which we have incorporated into this revision. }
\address{Mathematical Institute \\ Andrew Wiles Building \\ Radcliffe Observatory Quarter \\ Woodstock Rd \\ Oxford OX2 6QW}
\email{ben.green@maths.ox.ac.uk}

\subjclass[2000]{Primary }

\begin{abstract}
We say that a finite subset of the unit sphere in $\R^d$ is \emph{transitive} if there is a group of isometries which acts transitively on it. We show that the width of any transitive set is bounded above by a constant times $(\log d)^{-1/2}$.

This is a consequence of the following result: If $G$ is a finite group and $\rho : G \rightarrow \U_d(\C)$ a unitary representation, and if $v \in \C^d$ is a unit vector, there is another unit vector $w \in \C^d$ such that 
\[  \sup_{g \in G} |\langle \rho(g) v, w \rangle| \leq (1 + c \log d)^{-1/2}.\]

These results answer a question of Yufei Zhao. An immediate consequence of our result is that the diameter of any quotient $S(\R^d)/G$ of the unit sphere by a finite group $G$ of isometries is at least $\pi/2 - o_{d \rightarrow \infty}(1)$.
\end{abstract}


\maketitle

\section{Introduction}\label{sec1}

Let $\O(\R^d)$ be the $d$-dimensional orthogonal group, and write $S(\R^d)$ for the unit sphere in $\R^d$. If $V$ is a complex inner product space (for example $\C^d$), write $\U(V)$ for the group of unitary transformations on $V$. Write $S(V) = \{v \in V : \Vert v \Vert = 1\}$ for the unit sphere of $V$. 

\begin{definition}\label{def-1}
Let $f_{\R}(d)$ be the smallest function such that, for all finite groups $G \leq \O(\R^d)$ and for all $v \in S(\R^d)$ there is some $w \in S(\R^d)$ such that $\sup_{g \in G} |\langle gv, w \rangle| \leq f_{\R}(d)$. Let $f_{\C}(d)$ be the smallest function such that, for all finite groups $G \leq \U(\C^d)$, and for all $v \in S(\C^d)$ there is some $w \in S(\C^d)$ such that $\sup_{g \in G} |\langle g v, w \rangle| \leq f_{\C}(d)$. \end{definition}

\emph{Remarks.} Note that $f_{\R}(d), f_{\C}(d)$ are well-defined since the unit spheres $S(\R^d)$, $S(\C^d)$ are compact. 

Observe also that $f_{\C}(d)$ is also the least function  such that for all finite groups $G$, all unitary representations $\rho : G \rightarrow \U(\C^d)$ and all $v \in S(\C^d)$ there is some $w \in S(\C^d)$ such that $\sup_{g \in G} |\langle \rho(g) v, w \rangle| \leq f_{\C}(d)$: simply apply Definition \ref{def-1} to $\rho(G) \leq \U(\C^d)$.

Our main theorem is as follows.

\begin{theorem}\label{mainthm}
We have $f_{\R}(d), f_{\C}(d) \ll \frac{1}{\sqrt{\log d}}$ for $d \geq 2$. \end{theorem}

This answers a question of Yufei Zhao\footnote{Personal communication, 2016; also presented at the Workshop on Additive Combinatorics, Harvard, October 2017.} in the affirmative.

The bound for $f_{\R}(d)$ may be formulated in an intuitive geometric way. We say that a finite subset $X \subset S(\R^d)$ is \emph{transitive} if there is a group of isometries which acts transitively on $X$. Colloquially, ``all points of $X$ look the same''. The \emph{width} of $X$ is the minimal length, over all $w \in S(\R^d)$, of the orthogonal projection of $X$ onto the vector $w$. We denote this by $\width(X)$. A transitive set whose linear span is $\R^d$ is the same thing as an orbit $Gv$, where $G$ is a finite subgroup of $O(\R^d)$ and $v$ is any point of $X$, and therefore the bound on $f_{\R}(d)$ may be stated in the following manner, asserting that ``transitive sets are almost flat''.  

\begin{theorem}\label{flat-transitive}
Let $X \subset S(\R^d)$ be a transitive set. Then $\width(X) \ll \frac{1}{\sqrt{\log d}}$.
\end{theorem}
\emph{Remarks.} Readers familiar with high-dimensional phenomena will of course note that as $d \rightarrow \infty$ ``most'' of $S(\R^d)$ is contained in (say) the slab $\{x \in \R^d : |x_1| \lessapprox 1/\sqrt{\log d}\}$. Whilst this observation is by no means enough to prove Theorem \ref{flat-transitive}, it will be essential in our argument: see Proposition \ref{random-prop}. \vspace*{8pt}

\emph{Diameter of quotient spaces\footnote{This section was added in July 2019, the article having originally been posted in February 2018. We thank Alexander Lytchak for helpful comments.}.} It has come to the author's attention that an immediate consequence of Theorem \ref{flat-transitive} is that the diameter of any quotient space $S(\R^d)/G$ ($G$ finite) is $\frac{\pi}{2} - o_{d \rightarrow \infty}(1)$. This answers \cite[Conjecture 5.4]{dgms} in the case of finite groups. In fact, we have the stronger result that any point of $S(\R^d)/G$ is at distance at least $\frac{\pi}{2} - o_{d \rightarrow \infty}(1)$ from some other point. In terms of previous work on $\diam(S(\R^d)/G)$ we note, in addition to the paper \cite{dgms} just mentioned, the work of Greenwald \cite{greenwald} (who established a bound $\diam(S(\R^d)/G) \geq \eps(d)$ for some $\eps(d) > 0$) and the very recent work of Gorodski, Lange, Lytchak and Mendes \cite{gllm}, who showed\footnote{The authors of \cite{gllm} and I were completely unaware of each other's work. Note that a large part of \cite{gllm} is concerned with \emph{infinite} $G$, about which we say essentially nothing.} that $\eps(d)$ may be taken independent of $d$. \vspace{8pt}

\emph{Sharpness.} Let us note that Theorems \ref{mainthm} and \ref{flat-transitive} are sharp up to a multiplicative constant. To see this over $\R$, consider the transitive subset $X \subset S(\R^d)$ consisting of all permutations of all vectors 
\[ \frac{1}{\sqrt{H_d}} (\pm 1, \pm \frac{1}{\sqrt{2}},\dots, \pm \frac{1}{\sqrt{d}}),\] where $H_d = \sum_{i = 1}^d \frac{1}{i}$. Then we have 
\[ \inf_{w \in S(\R^d)} \sup_{x \in X} |\langle x, w\rangle| \geq \frac{1}{\sqrt{H_d}}\inf_{\substack{w \in S(\R^d) \\ w_1 \geq w_2 \geq \dots \geq w_d \geq 0}} \sum_{i=1}^d \frac{w_i}{\sqrt{i}}.\]
However if $w_1 \geq w_2 \geq \dots \geq w_d \geq 0$ then
\[
\big( \sum_{i = 1}^d \frac{w_i}{\sqrt{i}} \big)^2 \geq \sum_{\substack{i,j \in \{1,\dots, d\} \\i \leq j}} \frac{w_i w_j}{\sqrt{ij}}  \geq \sum_{\substack{i,j \in \{1,\dots, d\} \\ i \leq j}} \frac{w^2_j}{j} = \sum_{j = 1}^d w_j^2,
\]
and so it follows that 
\[ \inf_{w \in S(\R^d)} \sup_{x \in X} |\langle x, w\rangle| \geq \frac{1}{\sqrt{H_d}} \gg \frac{1}{\sqrt{\log d}}.\]
A very similar example works over $\C$.\vspace{5pt}

In the proofs of our main theorems will be using basic representation theory, and therefore it is much more natural to work over $\C$ than over $\R$. The following simple argument gives a bound for $f_{\R}(d)$ in terms of $f_{\C}(d)$, thus reducing the proof of Theorem \ref{mainthm} to the complex case.

\begin{lemma} We have $f_{\R}(d) \leq 2^{1/2} f_{\C}(d)$.
\end{lemma}
\begin{proof}
Let $G \leq \O(\R^d)$ and suppose that $v \in S(\R^d)$. Extend scalars to $\C$, thus regarding $v$ as an element of $\C^d$ and $G$ as a subgroup of $\U(\C^d)$. Pick some $w \in S(\C^d)$ with $\sup_{g \in G} |\langle gv, w\rangle| \leq f_{\C}(d)$.
Note that $|\langle gv, \Re w\rangle|, |\langle gv, \Im w\rangle| \leq |\langle gv, w\rangle|$ for all $g$, since the $gv$ are all real. By Pythagoras, one of $\Vert \Re w \Vert, \Vert \Im w \Vert$ is at least $\frac{1}{\sqrt{2}}$. Taking $x := \Vert \Re w \Vert^{-1} \Re w$ or $\Vert \Im w \Vert^{-1} \Im w$ as appropriate, we have $x \in S(\R^d)$ and $\sup_{g \in G} |\langle gv, x \rangle| \leq 2^{1/2}f_{\C}(d)$.  \end{proof}

The trivial bound for $f_{\C}(d)$ is $1$. Using Jordan's theorem, one can improve this slightly. 

\begin{proposition}\label{jordan-consequence}
For each $d \geq 2$ there is some $\eta_d > 0$ such that $f_{\C}(d) \leq 1 - \eta_d$. 
\end{proposition}
\begin{proof}
Jordan's theorem states that there is a function $F : \N \rightarrow \N$ with the following property. if $G \leq \U(\C^d)$ is finite then there is an abelian normal subgroup $A \lhd G$ with $[G : A] \leq F(d)$. By linear algebra, we may apply a unitary change of basis to $\C^d$ and thereby assume that every $g \in A$ is diagonal, with entries having absolute value $1$. Let $g_1,\dots, g_k$, $k \leq F(d)$, be a complete set of coset representatives for $A$ in $G$. We look first at $f_{\C}(d)$. For any $w \in \C^d$, 
\[ \sup_{g \in G} |\langle g v, w\rangle|  \leq \sup_{i = 1,\dots, k} \langle | g_i v|, |w| \rangle ,\] where $|z| \in S(\R^d)$ denotes the vector obtained from $z$ by replacing each coordinate with its absolute value. Note that the function from $S(\C^d)^k$ to $\R$ defined by $(v_1,\dots, v_k) \mapsto \inf_{w \in S(\C^d)}\sup_i |\langle |v_i|, |w |\rangle|$ is continuous, and so by compactness we need only show that this function is never $1$. If this were so, we would have $S(\C^d) \subset \bigcup_{i = 1}^k\{ w \in \C^d : |w| = |v_i|\}$ for some choice of $v_1,\dots, v_k$. However, the right-hand side is a union of $k$ real tori of dimension $d$, whereas the left-hand side is a sphere of dimension $2d - 1$ over $\R$. This is a contradiction if $d \geq 2$.
\end{proof}

M.~Collins \cite{collins1, collins2} has obtained the best possible bounds for $F(d)$, namely $F(d) \leq (d + 1)!$ (for large $d$).  However, even with this result to hand, any bound for $f_{\C}(d)$ obtained from Jordan's theorem alone tends to $1$ as $d \rightarrow \infty$. Indeed, a randomly chosen set of $(d+1)!$ points on $S(\C^d)$ will have width tending to $1$ as $d \rightarrow \infty$. 

We turn now to an overview of the proof of Theorem \ref{mainthm} in the complex case. A natural first thing to try is to choose $w$ \emph{randomly} on $S(\C^d)$ using the normalised Haar measure. Unfortunately, this does not always work. One may already see this with as uncomplicated a group as $G = (\Z/2\Z)^d$, where the orbit $Gv$ is the set of all $(\pm v_1 \pm v_2 ,\dots, \pm v_d)$. Taking $v = d^{-1/2} \mathbf{1}$ (where $\mathbf{1}$ is the vector all of whose entries are $1$), by choosing signs appropriately we see that for any $w$ we have $\sup_{g \in G} |\langle gv, w\rangle| \geq \frac{1}{\sqrt{d}}\sum_{i =1}^d |\Re w_i|$. However, the distribution of each $w_i$ is roughly Gaussian with mean $0$ and variance $\frac{1}{d}$, so for almost every $w \in S(\C^d)$ we have $\sum_{i =1}^d |\Re w_i| \gg \sqrt{d}$.

However, if we allow $w$ to be sampled according to more general probability measures on $S(\C^d)$ (depending only on $G$, and not on $v$) then it turns out that progress is possible. With this in mind, we make the following definition.

\begin{definition}\label{def-2}
Let $\tildef(d)$ be the smallest function such that the following is true. For every finite group $G \leq \U_d(\C)$, there is a probability measure $\mu$ on $S(\C^d)$ such that 
\begin{equation}\label{sym-measure}  \int \sup_{g \in G} |\langle gv, w\rangle|^2 d\mu(w) \leq \tildef(d)^2\end{equation} for all $v \in S(\C^d)$.
\end{definition}

Note that $f(d)$ is well-defined because the space of probability measures on $S(\C^d)$ is closed under weak limits. Let us reiterate that $\mu$ can, and will, depend on $G$, but does not depend on $v$. 

The following is immediate.

\begin{lemma}\label{triv-lem} We have $f_{\C}(d) \leq \tildef(d)$.\end{lemma}
\begin{proof}
Suppose we have a finite group $G \leq \U_d(\C)$ and some $v \in S(\C^d)$. Then, by definition, 
\[ \int \sup_{g \in G} |\langle gv, w\rangle|^2 d\mu(w) \leq \tildef(d)^2.\]
In particular, since $\mu$ is a probability measure, there exists some $w$ such that
\[ \sup_{g \in G} |\langle gv, w\rangle|^2  \leq \tildef(d)^2.\] This proves the result. 
\end{proof}

An important aspect of our proof is that it proceeds by induction, both on the integer $d$ and in the representation-theoretic sense, from subgroups\footnote{Implicitly; we use instead the language of ``systems of imprimitivity', and can avoid explicit discussion of induced representations.} of $G$. It does not seem to be possible to make such an argument work with $f_{\C}(d)$ directly, but the quantity $f(d)$ is well-suited to this approach, and moreover behaves well with regard to tensor products (see Section \ref{sec4}).

We will establish the following bound on $f(d)$, which turns out to be in a convenient form for our inductive argument.

\begin{theorem}\label{mainthm2}
For some absolute $c > 0$ we have $\tildef(d) \leq (1 + c \log d)^{-1/2}$ for all $d \geq 1$.
\end{theorem}

In view of Lemma \ref{triv-lem}, this immediately implies Theorem \ref{mainthm}. \vspace{8pt}

\emph{The primitive case.} There is one case in which it is not possible to proceed inductively. 
\begin{definition}
Let $G \leq \U_d(\C)$. Then we say that $G$ is \emph{imprimitive} if there is a \emph{system of imprimitivity} for $G$, that is to say a nontrivial direct sum decomposition $\C^d = \bigoplus_{i = 1}^k V_i$ (with nontrivial meaning that $0 < \dim V_i < d$) such that the $V_i$ are permuted by the action of $G$. If $G$ is not imprimitive then we say that it is primitive.
\end{definition}

\emph{Remarks.} Primitivity, like irreducibility, is not an intrinsic property of the group $G$ but rather of the action of $G$ on $\C^d$ (in other words, it is a property of \emph{representations} rather than groups).

Note in particular that if $G$ is primitive then it is irreducible (that is, no proper subspace of $\C^d$ is fixed by $G$). Indeed, if $G$ is reducible then there is a proper subspace $V$ which is fixed by $G$. Its orthogonal complement $V^{\perp}$ will also be fixed by $G$, and so the action of $G = V \oplus V^{\perp}$ is a system of imprimitivity for $G$.

\begin{definition}
We define $\tildef_{\prim}(d)$ in the same way as $\tildef(d)$ (Definition \ref{def-2}) but with $G$ ranging only over \emph{primitive} groups $G \leq \U_d(\C)$. 
\end{definition}

A large part of our paper will be devoted to the primitive case, that is to say the proof of the following proposition.

\begin{proposition}\label{mainprop}
For some absolute $c > 0$ we have $\tildef_{\prim}(d) \leq (1 + c \log d)^{-1/2}$ for all $d \geq 1$.
\end{proposition}

As already stated, induction is of no use here. Instead, we must analyse the structure of primitive subgroups of $\U_d(\C)$. Here we use some arguments of Collins \cite{collins1}, controlling such groups by their generalised Fitting subgroups $F^*(G)$ (we will give the definitions later). In the ensuing analysis we must, unfortunately, make an appeal to the Classification of Finite Simple Groups (CFSG): roughly, we need to know that the alternating groups are the only nonabelian finite simple groups $\Gamma$ with a nontrivial representation of degree $\leq (\log |\Gamma|)^{O(1)}$. (We also need some other consequences.)

 \vspace{8pt}

\emph{Permutation groups.} Permutations arise in both the imprimitive case (because imprimitive groups permute the summands of a system of imprimitivity) \emph{and} as an example of the primitive case (because the action of the symmetric group $S_d$ on $\mathbf{1}^{\perp} = \{z \in \C^d : z_1 + \dots + z_d = 0\}$ is primitive).

 Write $\Gamma_d \leq \U_d(\C)$ for the (infinite) group consisting of all permutation matrices with entries of absolute value $1$. Thus the orbit $\Gamma_d v$ consists of all vectors $(\lambda_1 v_{\pi(1)},\dots, \lambda_d v_{\pi(d)})$ with $|\lambda_1| = \dots = |\lambda_d| = 1$. Then we define $\tildef_{\perm}(d)$ as for $\tildef(d)$, but with $G$ replaced by $\Gamma_d$:

\begin{definition}\label{def-3}
Let $\tildef_{\perm}(d)$ be the smallest function such that the following is true. There is a probability measure $\mu_{\Gamma_d}$ on $S(\C^d)$ such that 
\[  \int \sup_{g \in \Gamma_d} |\langle gv, w\rangle|^2 d\mu_{\Gamma_d}(w) \leq \tildef_{\perm}(d)^2\] for all $v \in S(\C^d)$.
Written explicitly, 
\[  \int \sup_{\pi \in S_d} \big(\sum_{i = 1}^d |v_{\pi(i)} w_{i}| \big)^2d\mu_{\Gamma_d}(w) \leq \tildef_{\perm}(d)^2\] for all $v \in S(\C^d)$.
\end{definition}

\begin{proposition}\label{permprop}
We have $\tildef_{\perm}(d) \leq (1 + c \log d)^{-1/2}$ for some absolute constant $c > 0$. 
\end{proposition}
The proof of this is given in Section \ref{perm-sec}. The following variant, in which we can (with a negligible cost) assume that $\mu_{\Gamma_d}$ is supported on $\mathbf{1}^{\perp} \subset \C^d$, is needed in the analysis of the primitive case.

\begin{proposition}\label{permprop-2}
There is a probability measure $\mu^*_{\Gamma_d}$ on $S(\C^d)$, supported on $\mathbf{1}^{\perp}$, such that 
\[ \int \sup_{\pi \in S_d} \big(\sum_{i = 1}^d |v_{\pi(i)} w_i | \big)^2 d \mu^*_{\Gamma_d}(w) \ll \frac{1}{\log d} \Vert v \Vert^2\] for all $v \in \C^d$.
\end{proposition}

The proof of this is also given in Section \ref{perm-sec}. \vspace{8pt}

\emph{Representation theory.} In this paper, ``representation'' will always mean (finite-dimensional) unitary representation, or in other words a homomorphism $\rho : G \rightarrow \U(V)$ into the space of unitary endomorphisms of some finite-dimensional hermitian space $V$. Of course, by Weyl's unitary trick every complex representation is equivalent to a unitary one. A somewhat less well-known fact to which we will appeal is that ``equivalent unitary representations are unitarily equivalent'': see Lemma \ref{unitary-equivalence} for a statement and proof in the irreducible case. This allows us to operate entirely within the world of unitary representations.

\section{Induction on dimension and on the group}

The aim of this section is to carry out the induction procedure discussed in the introduction, deducing Theorem \ref{mainthm2} (and hence Theorem \ref{mainthm}) from Propositions \ref{mainprop} and \ref{permprop}. The proof of these two propositions will occupy the remainder of the paper. 

The key result is Proposition \ref{inductive-prop} below. In this proposition, we define $\tildef_{\irred}(d)$ in the same way as $\tildef(d)$ (Definition \ref{def-2}) but with $G$ ranging only over groups acting \emph{irreducibly} on $\U_d(\C)$. 

\begin{proposition}\label{inductive-prop} We have
\begin{equation}\label{eq-ind} \tildef_{\irred}(d) \leq  \max_{\substack{ d_1d_2 = d }} f_{\perm}(d_1)  f_{\prim}(d_2) ,\end{equation}
where $d_1$ and $d_2$ range over positive integers.
\end{proposition}

We begin by showing that this does provide an inductive proof of our main theorem, Theorem \ref{mainthm2}, given Propositions \ref{mainprop} and \ref{permprop}.

\begin{proof}[Proof of Theorem \ref{mainthm2}, assuming Propositions \ref{mainprop}, \ref{permprop} and \ref{inductive-prop}.]
Let $c$ be the constant appearing in Propositions \ref{mainprop} and \ref{permprop}. For real $x \geq 1$, set $\eta(x) := (1 + c\log x)^{-1/2}$. Observe the inequality
\begin{equation}\label{eta-2}
\eta(x) \eta(y) \leq \eta(xy)
\end{equation}
for all $x, y \geq 1$. Indeed, a short manipulation shows that this is equivalent to
\[ 1 + c \log x + c\log y \leq (1 + c \log x)(1 + c \log y),\] which is immediate (for any $c> 0$) upon expanding out the right hand side. 

It follows immediately from this (assuming Propositions \ref{mainprop}, \ref{permprop} and \ref{inductive-prop}) that 
\begin{equation}\label{irred-case} \tildef_{\irred}(d) \leq \eta(d) = (1 + c \log d)^{-1/2}.\end{equation}
This establishes Theorem \ref{mainthm2} in the case that $G$ acts irreducibly. We now show that the general case follows from the irreducible one by proving the following inequality, valid for any positive integer $m$:
\begin{equation}\label{irred-ineq} \tildef(d) \leq \max (\frac{1}{\sqrt{m}}, \max_{d' \geq d/m} \tildef_{\irred} (d')).\end{equation}
Taking $m = \log^{10} d$ (say), one immediately sees that the general case of Theorem \ref{mainthm2} indeed follows from \eqref{irred-case} (with a correction to lower order terms, but essentially the same constant $c$).

It remains to prove \eqref{irred-ineq}. I thank Ashwin Sah, Mehtaab Sawhney and Yufei Zhao for showing me \eqref{irred-ineq} and its proof, which replaces an incorrect argument in an earlier version of this paper.  

The key observation is that (by Maschke's theorem) $\C^d$ has a decomposition $\bigoplus V_i$ into orthogonal irreducible representations of $G$. There are either at least $m$ of them (\emph{case 1}) or one of them has dimension at least $d/m$ (\emph{case 2}). 

In case 1, we proceed as follows. For $i \leq m$, pick an arbitrary unit vector $w_i$ in $V_i$,  and let $\mu = \frac{1}{m} \sum_{i =1}^m \delta_{w_i}$, where $\delta_w$ is the measure with mass $1$ at the point $w$. Let $v \in \C^d$ be a unit vector, and suppose that $v = \sum v_i$ is the decomposition of $v$, with $v_i \in V_i$. 

If $g \in G$ then we have $\langle gv, w_i\rangle = \langle g v_i, w_i\rangle$ by orthogonality, and therefore 
\[ \sup_{g \in G} |\langle gv, w_i\rangle |^2 \leq \sup_{g \in G} \Vert g v_i \Vert^2 \Vert w_i \Vert^2 \leq \Vert v_i \Vert^2.\]
It follows that 
\[ \int \sup_{g \in G} |\langle gv, w \rangle|^2 d \mu(w) \leq \frac{1}{m} \sum_{i = 1}^m \Vert v_i \Vert^2 \leq \frac{1}{m}.\] This is the first of the bounds in \eqref{irred-ineq}.

In case 2, suppose that $V_1$ is an irreducible summand of $V$ with maximal dimension $d_1 \geq d/m$.
Let $\mu_1$ be a measure on unit vectors in $V_1$ such that 
\[ \int \sup_{g \in G} |\langle gv_1, w_1\rangle |^2 d\mu_1(w_1) \leq \tildef_{\irred}(d_1).\]
Let $\mu$ be the pushforward of $\mu_1$ under the inclusion map $\iota : V_1 \rightarrow V$. 

Let $v = \sum v_i$ be a unit vector in $V$. Then
\begin{align*} \int \sup_{g \in G} |\langle g v, w \rangle |^2 d\mu(w) & = \int \sup_{g \in G} |\langle g v, w_1 \rangle |^2 d\mu_1(w_1) \\ & = \int \sup_{g \in G} |\langle g v_1, w_1 \rangle |^2 d\mu_1(w_1) \\ & \leq \tildef_{\irred}(d_1) \Vert v_1 \Vert^2 \\ & \leq \max_{d' \geq d/m} \tildef_{\irred}(d'),\end{align*}
where in the last step we used that $\Vert v_1 \Vert^2 \leq 1$.

This completes the proof of \eqref{irred-ineq}, and hence of Theorem \ref{mainthm2}, assuming Propositions \ref{mainprop}, \ref{permprop} and \ref{inductive-prop}.
\end{proof}

The remainder of the section is devoted to establishing Proposition \ref{inductive-prop}.

\begin{proof}[Proof of Proposition \ref{inductive-prop}]
Let $G \leq \U_d(\C)$. Noting that the right-hand side of \eqref{eq-ind} is at least $\tildef_{\prim}(d)$ (since $f_{\perm}(1) = 1$), it suffices to deal with the case in which $G$ is not primitive. Assume, then, that $G$ is irreducible, but not primitive. Let $\C^d = \bigoplus_{i = 1}^{d_1} V_i$ be a system of imprimitivity. I feel that the following lemma must surely be known in the literature (and probably goes back to Frobenius), but I do not know\footnote{I asked for a reference for this on Math Overflow \cite{overflow-post}, but none has so far been forthcoming. Paul Broussous did provide an alternative (though related) proof similar to that of Lemma \ref{unitary-equivalence} which some readers may find more natural.} a source.

\begin{lemma}\label{orth}
Let $G \leq \U_d(\C)$ be irreducible and imprimitive, with $\C^d = \bigoplus_{i = 1}^r V_i$ a system of imprimitivity. Then the $V_i$ are orthogonal.
\end{lemma}
\begin{proof}
Write $\pi_i : \C^d \rightarrow V_i$ for the projection induced by the direct sum decomposition $\C^d = \bigoplus_{i = 1}^r V_i$, thus $x = \sum_{i = 1}^r \pi_i(x)$ for all $x$. Define a map $\phi : \C^d \rightarrow \C^d$ by $\phi(x) = \sum_{i = 1}^r \pi_i^* \pi_i (x)$. Note that if $x \in V_i$ and $y \in V_j$ then
\begin{equation}\label{eq109} \langle \phi(x), y \rangle = \langle \pi_i^* \pi_i (x), y \rangle = \langle \pi_i(x) , \pi_i(y)\rangle = 1_{i = j} \langle x, y\rangle.\end{equation}
In particular, $\phi$ is not identically zero.
We claim that $\phi$ is $G$-equivariant, that is to say 
\begin{equation}\label{eq110} \phi(gx) = g \phi(x).\end{equation} To prove this, it suffices by bilinearity to show that 
\begin{equation}\label{eq111}
\langle \phi(gx), y\rangle = \langle g\phi(x), y\rangle
\end{equation}
whenever $x \in V_i, y \in V_j$. Since $g$ permutes the $V_i$, we have $gx \in V_{\sigma_g(i)}$, $g^{-1} y \in V_{\sigma_{g^{-1}}(j)}$ for some mutually inverse permutations $\sigma_g, \sigma_{g^{-1}}$. Therefore, by \eqref{eq109}, we have
\[ \langle \phi(g x), y\rangle = 1_{\sigma_g(i) = j} \langle gx, y\rangle,\] whilst
\[ \langle g \phi(x), y\rangle = \langle \phi(x), g^{-1} y\rangle = 1_{i = \sigma_{g^{-1}}(j)} \langle x, g^{-1} y\rangle = 1_{\sigma_g(i) = j} \langle gx, y\rangle.\]
This establishes \eqref{eq111} and thus $\phi$ is indeed $G$-equivariant. By Schur's lemma and the irreducibility of $G$, we must have $\phi(x) = \lambda x$ for some non-zero scalar $\lambda \in \C$. Therefore if $x \in V_i$ and $y \in V_j$ with $i \neq j$ we have, from \eqref{eq109},
\[ \lambda \langle x, y \rangle = \langle \phi(x), y\rangle = 1_{i = j} \langle x, y \rangle = 0.\]
It follows that $V_i$ is indeed orthogonal to $V_j$.
\end{proof}

We continue with the proof of Proposition \ref{inductive-prop}. Note that, for any $D$, $G$ fixes $\bigoplus_{i : \dim V_i = D} V_i$; since $G$ is irreducible, it follows that all the $V_i$ have the same dimension $d_2$, and hence that $d_1d_2 = d$. Let us further assume that $d_2$ is minimal, over all such systems of imprimitivity. Let 
\[ H := \{g \in G : gV_1 = V_1\}.\] 
If $G$ does not act transitively on the $V_i$ then it is again reducible, an invariant subspace being thelinear span of the orbit $G V_1$. Therefore there are $\gamma_1,\dots, \gamma_{d_1} \in G$ such that $\gamma_i V_1 = V_i$ for $i = 1,\dots, d_1$. In fact, $\{\gamma_1,\dots, \gamma_{d_1}\}$ has this property if and only if it is a complete set of left coset representatives for $H$ in $G$.

We claim that the action of $H$ on $V_1$ is primitive. If not, there is a nontrivial orthogonal system of imprimitivity 
\[ V_1 = \bigoplus_{j = 1}^{\ell} W_j\] for $H$, thus the $W_j$ are permuted by $H$. Set $V_{ij} := \gamma_{i} W_j$.

Note that if $g \in G$ then $\{g\gamma_1,\dots, g\gamma_{d_1}\}$ is a complete set of left coset representatives for $H$ in $G$, and so there is a permutation $\sigma_g \in S_{d_1}$ with the property that $\gamma_{\sigma_g(i)}^{-1} g \gamma_i \in H$ for all $i = 1,\dots, d_1$. Write $h(g,i) = \gamma_{\sigma_g(i)}^{-1} g \gamma_i$. 

For $h \in H$, write $\pi_h \in S_{\ell}$ for the permutation such that $h W_j = W_{\pi_h(j)}$ for all $j$.
Then
\[ g V_{ij} = g \gamma_i W_j = \gamma_{\sigma_g(i)} h(g,i) W_j = \gamma_{\sigma_g(i)} W_{\pi_{h(g,i)} j} = V_{\sigma_g(i), \pi_{h(g,i)} j}.\]
Thus the $V_{ij}$ are permuted by the action of $G$. Moreover, any two distinct $V_{ij}, V_{i'j'}$ are orthogonal: if $i \neq i'$ then this is clear, since $V_{ij} \subset V_i$ and $V_{i'j'} \subset V_{i'}$, whilst if $i = i'$ but $j \neq j'$ then $V_{ij} = \gamma_i W_j$ and $V_{i'j'} = \gamma_{i} W_{j'}$ are orthogonal since $W_j, W_{j'}$ are and $\gamma_i$ is unitary.

It follows that $\bigoplus_{i,j} V_{ij}$ is an orthogonal system of imprimitivity for $G$. Since $\dim V_{1,1} < d_2$, this contradicts the minimality of $d_2$. Therefore we were wrong to assume that the action of $H$ was imprimitive.

Let $\mu_1$ be the probability measure on $S(V_1)$ guaranteed by Definition \ref{def-1}, that is to say
\begin{equation}\label{muh} \int \sup_{h \in H} |\langle hv, x\rangle|^2 d\mu_1(x) \leq \tildef_{\prim}(d_2)^2 \Vert v \Vert^2\end{equation} for all $v \in V_1$. Let $\mu_{2}$ be the probability measure on $S(\C^{d_1})$ guaranteed by Definition \ref{def-3}, that is to say
\begin{equation}\label{musym}  \int \sup_{\pi \in S_{d_1}} \sum_{i = 1}^{d_1} |\lambda_i a_{\pi(i)}|^2 d\mu_{2}(\lambda) \leq \tildef_{\perm}(d_1)^2 \Vert a \Vert^2\end{equation} for all $a \in \C^{d_1}$.
Define $\mu$ to be the pushforward of $\mu_1 \times \mu_{2}$ under the map
\[ \psi: S(V_1) \times S(\C^{d_1}) \rightarrow S(\C^d)\] defined by
\[ \psi(x, \lambda) = \sum_{i = 1}^{d_1} \lambda_i \gamma_i x.\]
(Note that, since the $\gamma_i x$, $i = 1,\dots, d_1$ are orthogonal vectors in $\C^d$, all with norm $\Vert x \Vert$, $\psi$ does indeed take values in $S(\C^d)$).

Now let $v \in \C^d$. We have a unique orthogonal decomposition
\[ v = \sum_{i = 1}^{d_1} \gamma_i v_i,\] where $v_i \in V_1$. For $g \in G$, we have
\begin{equation}\label{gv} gv   = \sum_{i = 1}^{d_1} g \gamma_i v_i  = \sum_{i = 1}^{d_1} \gamma_{\sigma_g(i)} h(g,i) v_i  = \sum_{i = 1}^{d_1} \gamma_i h(g, \sigma_g(i)) v_{\sigma_g^{-1}(i)}.\end{equation} (The definitions of $\sigma_g$ and $h(g,i)$ were given above.)
Let 
\[ w = \sum_{i = 1}^{d_1} \lambda_i \gamma_i x,\] where $x \in S(V_1)$ and $\lambda \in S(\C^{d_1})$.
It follows from \eqref{gv}, the orthogonality of the $\gamma_i V_1$ and the unitary nature of the $\gamma_i$ that 
\[ \langle gv, w\rangle = \sum_{i = 1}^{d_1} \lambda_i \langle h(g, \sigma_g(i)) v_{\sigma_g^{-1}(i)}, x\rangle.\]
Therefore
\[ \sup_{g \in G} |\langle gv, w\rangle| \leq \sup_{\pi \in S_{d_1}} \sum_{i = 1}^{d_1} |\lambda_i| \sup_{h \in H} |\langle h v_{\pi(i)}, x\rangle|.\]
Squaring and integrating with respect to $\mu$, we have
\begin{align*} \int & \sup_{g \in G} |\langle gv, w \rangle|^2 d\mu(w) \\ & \leq \int  \int \sup_{\pi \in S_{d_1}}\big( \sum_{i = 1}^{d_1} |\lambda_i| \sup_{h \in H} |\langle h v_{\pi(i)}, x\rangle| \big)^2  d\mu_{2}(\lambda) d\mu_1(x)\\ & \leq \tildef_{\perm}(d_1)^2\int  \sum_{i = 1}^{d_1}\sup_{h \in H}|\langle h v_{i}, x\rangle|^2 d\mu_1(x) \\ & \leq \tildef_{\perm}(d_1)^2 \tildef_{\prim}(d_2)^2\sum_{i = 1}^{d_1} \Vert v_i \Vert^2 = \tildef_{\perm}(d_1)^2 \tildef_{\prim}(d_2)^2\Vert v\Vert^2,\end{align*}
where in these last two lines we used \eqref{musym} (with the choice $a_i = \sup_{h \in H}|\langle h v_{i}, x\rangle|$) and \eqref{muh} respectively.

This is the second bound in \eqref{eq-ind}, and so the proof of Proposition \eqref{inductive-prop} is complete as we have now covered all cases.
\end{proof}

\section{Permutation groups}\label{perm-sec}

In this section we establish Propositions \ref{permprop} and \ref{permprop-2}.

\begin{proof}[Proof of Proposition \ref{permprop}]
We define the measure $\mu_{\Gamma_d}$ on $S(\C^d)$ very explicitly. Let $m := \lceil \frac{\log d}{2\log 2}\rceil$, and  for $i = 0,1,\dots, m$ define $e_i \in S(\C^d)$ to be the vector whose first $2^i$ coordinates equal $2^{-i/2}$, and whose remaining coordinates are zero. Then define 
\[ \mu_{\Gamma_d}  := \frac{1}{2(m+1)}\sum_{i = 0}^m (\delta_{e_i} + \delta_{-e_i}).\]
Here, $\delta_x$ is the Dirac measure at the point $x$, defined by $\int f d\delta_x = f(x)$.
Suppose that $v \in S(\C^d)$. Let $v' \in \Gamma_d v$ have all its coordinates real and $v'_1 \geq \dots \geq v'_d \geq 0$ (this is possible by the definition of $\Gamma_d$). Then, for each $i$,
\[
\sup_{g \in \Gamma_d} |\langle gv, e_i\rangle|  = \sup_{g \in \Gamma_d} |\langle g v', e_i \rangle|  = |\langle v', e_i\rangle|.
\] 
It follows that 
\begin{equation}\label{to-use} \sup_{v \in S(\C^d)} \int \sup_{g \in \Gamma_d} |\langle gv, w\rangle|^2 d\mu_{\Gamma_d}(w) \leq  \sup_{v \in S(\C^d)} \frac{1}{m+1} \sum_{i = 0}^m |\langle v, e_i \rangle|^2.\end{equation}
Were the $e_i$ orthonormal, it would be a trivial matter to bound the right-hand side using Bessel's inequality (by $\frac{1}{m+1}$). Whilst this is manifestly not the case, the $e_i$ are, in a sense, \emph{almost} orthonormal. In such a situation an inequality of Selberg (see, for example, \cite[p14]{grandcrible} or \cite[\S 27, Theorem 1]{davenport}) can take the place of Bessel's inequality. This states (for any choice of $e_0,\dots, e_m$, not just our specific one) that for $v \in S(\C^d)$ we have
\[ \sum_{i = 0}^m |\langle v, e_i \rangle|^2 \leq \sup_i \sum_{j= 0}^m |\langle e_i, e_j\rangle|.\]
In our case it is easy to see that 
\[ \langle e_i, e_j \rangle = 2^{-|i - j|/2},\] and therefore
\[ \sup_i \frac{1}{m+1}\sum_{j= 0}^m |\langle e_i, e_j\rangle|  = \frac{1}{m+1}\max_{\substack{I \subset \Z \\ |I| = m+1}}\sum_{n \in I} 2^{-|n|/2} := \psi(d),\]say,  where the maximum is over all discrete intervals $I$ of length $m+1$. Thus, by Selberg's inequality,
\begin{equation}\label{selb}\sum_{i = 0}^m |\langle v, e_i \rangle|^2 \leq \psi(d).\end{equation}
If $d \geq 2$ then $m \geq 1$, and so it is clear from the definition of $\psi$ that $\psi(d) < 1$. Moreover, for all $d$ we have
\[ \psi(d) \leq \frac{1}{m+1} \sum_{n \in \Z} 2^{-|n|/2} = \frac{3 + 2\sqrt{2}}{m+1} \leq \frac{2(3 + 2\sqrt{2})\log 2}{\log d}.\] 
From these two properties and the fact that $\psi(1) = 1$ it follows that, if $c$ is small enough, $\psi(d) \leq (1 + c\log d)^{-1/2}$ for all $d$.

Therefore \eqref{selb} and \eqref{to-use} imply that 
\[ \sup_{v \in S(\C^d)} \int \sup_{g \in \Gamma_d} |\langle gv, w\rangle|^2 d\mu_{\Gamma_d}(w) \leq  \psi(d) \leq (1 + c\log d)^{-1/2}.\] This concludes the proof.
\end{proof}

\emph{Remark.} It would be interesting to determine the best possible choice of measure as regards this proposition. 

\begin{proof}[Proof of Proposition \ref{permprop-2}.]  With notation as before, define
\[ \mu^*_{\Gamma_d} := \frac{1}{2(m+ 1)} \sum_{i = 0}^{m} (\delta_{\overline{e_i}} + \delta_{-\overline{e_i}}),\] where $\overline{e}_i$ is the orthogonal projection of $e_i$ to the subspace $\mathbf{1}^{\perp}$, that is to say
\[ \overline{e}_i = e_i - \frac{2^{i/2}}{d} \mathbf{1},\] where $\mathbf{1} \in \C^d$ is the vector of $d$ 1's. For each $i$ we have
\[ \Vert \overline{e}_i - e_i \Vert = \frac{2^{i/2}}{\sqrt{d}} = O(d^{-1/4})\] since $i \leq m$. 
Thus for each $i$, for each $v \in S(\C^d)$ and for each $g \in \Gamma_d$, we have
\[ |\langle gv, \overline{e}_i \rangle| \leq |\langle gv, e_i \rangle | + O(d^{-1/4}),\] and hence
\[ |\langle gv, \overline{e}_i \rangle|^2 \leq |\langle gv, e_i \rangle|^2 + O(d^{-1/4}), \] and therefore
\[ \sup_{g \in \Gamma_d}|\langle gv, \overline{e}_i \rangle|^2 \leq \sup_{g \in \Gamma_d}|\langle gv, e_i \rangle|^2 + O(d^{-1/4}). \] It follows from this and Proposition \ref{permprop} that we have
\begin{align*} \int \sup_{g \in \Gamma_d}|\langle gv, w \rangle|^2 d\mu^*_{\Gamma_d}(w) & \leq \int \sup_{g \in \Gamma_d} |\langle gv, w\rangle|^2 d\mu_{\Gamma_d}(w) + O(d^{-1/4}) \\ & \ll \frac{1}{\log d},\end{align*} as required.
\end{proof}

\section{The primitive case: reduction to an inverse theorem} \label{sec4}We now begin to turn our attention towards the most substantial task in the paper, which is to prove Proposition \ref{mainprop}, or in other words our main theorem in the primitive case. We begin with a relatively simple result, valid for all $G \leq \U_d(\C)$, not just primitive $G$. Here, and in subsequent places, we write $Z_d := \{\lambda I_d : |\lambda| = 1\} \subset \U(\C^d)$.

\begin{proposition}\label{random-prop}
Suppose that $G \leq \U_d(\C)$ is a finite group and that $[G : Z_d \cap G] \leq e^{d/\log d}$. Then there is a probability measure $\mu$ on $S(\C^d)$ such that 
\[ \sup_{v \in S(\C^d)} \int \sup_{g \in G} |\langle gv, w \rangle|^2 d\mu(w) \ll  \frac{1}{\log  d}.\]
\end{proposition}
\begin{proof}
Take $\mu$ to be the normalised Haar measure on $S(\C^d)$. Let $g_1,\dots$, $g_m$, $m \leq e^{d/\log d}$, be a complete set of coset representatives for $Z_d \cap G$ in $G$. Then
\[ \sup_{g \in G} |\langle gv, w\rangle | = \sup_{i = 1,\dots, m} |\langle g_i v, w \rangle|,\] and so
\begin{equation}\label{to-bound-4} \sup_{v \in S(\C^d)} \int \sup_{g \in G} |\langle gv, w \rangle|^2 d\mu(w) \leq \sup_{v^1,\dots, v^m \in S(\C^d)} \int \sup_i |\langle v^i , w\rangle |^2d\mu(w).\end{equation}
By standard estimates for the volume of spherical caps (see \cite[Lemma 2.2]{ball} for a beautiful exposition) we have, for each $i$ and for any\footnote{In fact the estimate is true for all $0 < t \leq 1$, but the simple geometric argument given in \cite{ball} is only valid in this more restricted range.} $t < 1/\sqrt{2}$), 
\[ \int 1_{ |\langle v^i, w \rangle| > t } d\mu(w) \leq 2e^{-t^2 d/2}.\] By the union bound, it follows that for any $t < 1/\sqrt{2}$ we have
\[ \int \sup_i |\langle v^i , w\rangle |^2d\mu(w) \leq t^2 + 2m e^{-t^2 d/2}.\]
Taking $t = 2/\log^{1/2} d$ gives the result. 
\end{proof}

The example of the symmetric group $G = S_{d+1}$ acting on $\C^d$ shows that the hypotheses of the proposition do not always hold: for this group we have $[G : Z_d \cap G] = (d + 1)! \sim e^{d \log d}$. However, Collins \cite{collins1} showed that, for $d$ large and for $G \leq \U_d(\C)$ primitive, this is the worst situation. To bridge the gap between $e^{d/\log d}$ (in the hypothesis of Proposition \ref{random-prop}) and $e^{d \log d}$ (from Collins' result) we need the following ``inverse theorem''.

\begin{proposition}\label{inverse-collins} Let $d$ be sufficiently large.
Let $G \leq \U^d(\C)$ be primitive and suppose that $[G : Z_d \cap G] \geq e^{d/\log d}$. Then $G$ has a normal subgroup isomorphic to the alternating group $A_n$, for some $n \gg d/\log^4 d$. \end{proposition}

The proof of Proposition \ref{inverse-collins} proceeds via a careful analysis of Collins's argument. Unfortunately, it is not by any means possible to extract this result directly from Collins's paper, so we must give a self-contained account of the argument. This task occupies the final section of the main part of the paper. Before turning to that, we show how Propositions \ref{random-prop} and \ref{inverse-collins} combine to yield a proof of Proposition \ref{mainprop}, leaving the proof of Proposition \ref{inverse-collins} as the only outstanding task.

\begin{proof}[Proof of Proposition \ref{mainprop}, assuming Proposition \ref{inverse-collins}]
By Proposition \ref{jordan-consequence}, we may assume that $d$ is sufficiently large. Therefore by Proposition \ref{random-prop}, and assuming Proposition \ref{inverse-collins}, it is enough to show that if $ G \leq \U_d(\C)$ has a normal subgroup isomorphic to $A_n$, $n \gg d/\log^4 d$, then there is a probability measure $\mu$ on $S(\C^d)$ such that 
\begin{equation}\label{claim-1} \sup_{v \in S(\C^d)} \int \sup_{g \in G} |\langle g v, w \rangle |^2 d\mu(w) \ll \frac{1}{\log d}.  \end{equation}
 
Since $d$ is sufficiently large we may assume that $n \geq 15$, so that the results of Proposition \ref{alt-prop} all hold. We will assume them without further comment. Write $A$ for the normal subgroup of $G$ which is isomorphic to $A_n$. The key observation now is that $G$ is in fact almost a direct product of $A$ and another group $H$, because $A_n$ is ``almost'' complete (trivial centre and outer automorphism group). Consider the natural map
\[ \pi : G \rightarrow \Aut(A) \times (G/A),\] where $g$ maps to the automorphism of $A$ given by conjugation by $g$. 
The kernel of this map is $Z(A)$, which is trivial, and so $\pi$ is injective.  On the other hand $\Aut(A) \cong S_n$, and so $|\Aut(A)| = 2|A|$. Therefore $\pi$ embeds $G$ as a subgroup of index 2 in the direct product $\Aut(A) \times (G/A)$. Let $G_1 := \{ g_1 \in \Aut(A): (g_1,1) \in \pi(G)\}$ and $G_2 = \{ g_2 \in G/A : (1, g_2) \in \pi(G)\}$. Then $G_1 \times G_2 \subset \pi(G)$, and $[\Aut(A) : G_1], [(G/A) : G_2] \leq 2$. It follows that $[\pi(G) : G_1 \times G_2] \leq 2$.  If $[\Aut(A) : G_1] = 2$ then $G_1 \cong A_n$, since $A_n$ is the unique index 2 subgroup of $S_n$. If $G_1 = \Aut(A)$ then we we must have $[(G/A) : G_2] = 2$ and $\pi(G) \cong \Aut(A) \times G_2 \cong S_n \times G_2$. In either case, $G$ has a subgroup of index at most $2$ which is isomorphic to a direct product of $A_n$ and another group $H$.

We now make a relatively simple deduction to remove the ``index at most 2'' issue, claiming that it is enough to establish \eqref{claim-1} in the case that $G$ is actually isomorphic to a direct product $A_n \times H$. Suppose that $G$ has a subgroup $G'$ of index $2$, isomorphic to $A_n \times H$, and that we have a measure $\mu'$ satisfying the analogue of \eqref{claim-1} with $G'$ in place of $G$. Set $\mu := \mu'$, and let $x \in G \setminus G'$. Then 
\begin{align*}
 \sup_{v \in S(\C^d)} \int \sup_{g \in G} & |\langle g v, w \rangle |^2 d\mu(w)  \\ & \leq \sup_{v \in S(\C^d)} \int \sup_{g' \in G'}\big( |\langle g'v, w \rangle| + |\langle g'xv, w\rangle| \big)^2 d\mu'(w) \\ & \leq 2\sup_{v \in S(\C^d)} \int \sup_{g' \in G'} |\langle g'v, w \rangle|^2 d\mu'(w) + \\ & \qquad\qquad 2\sup_{v \in S(\C^d)} \int \sup_{g' \in G'}|\langle g'xv, w\rangle| \big)^2 d\mu'(w) \\ & \ll \frac{1}{\log d}.
\end{align*}
Suppose, then, that $G \cong A_n \times H$. At this point, to avoid confusion, it is convenient to move to the language of representation theory and write $\rho : A_n \times H \rightarrow \U(\C^d)$ for the representation induced from the isomorphism between $A_n \times H$ and $G$, a subgroup of $\U(\C^d)$. Let $\rho = \bigoplus \rho_i$ be the decomposition of this representation into irreducible unitary subrepresentations (of $A_n \times H$); thus we have an orthogonal direct sum decomposition $\C^d = \bigoplus V_i$ with each $\rho_i : A_n \times H \rightarrow \U(V_i)$ being irreducible. Since $\rho$ is faithful, at least one of the $\rho_i$, say $\rho_1$, is nontrivial (and hence, since $A_n$ is simple, faithful) when restricted to $A_n$. By Lemma \ref{product-rep}, $\rho_1 \cong \psi \otimes \psi'$, where the isomorphism is one of $(A_n \times H)$-representations, and $\psi, \psi'$ are pullbacks of irreducible representations of $A_n, H$ respectively, and $\psi$ must be nontrivial (in fact, faithful). Note that $\dim \psi \leq \dim \rho_1 \leq \dim \rho = d < \frac{1}{2}n(n-3)$. By Lemma \ref{alt-prop} (2) it follows that $\psi$ is isomorphic to the $(n-1)$-dimensional permutation representation of $A_n$ on $X := \mathbf{1}^{\perp} \subset \C^n$. Note that $\psi$ is unitary with respect to the inner product restricted from the standard one on $\C^n$.  Suppose that $\psi'$ takes values in $\U(X')$ for some hermitian inner product space $X'$.

The tensor product $X \otimes X'$ has a natural unitary structure, with the inner product being defined by $\langle x_1 \otimes x'_1, x_2 \otimes x'_2 \rangle = \langle x_1, x'_1\rangle \langle x_2,x'_2\rangle$ on pure tensors. The representation $\psi \otimes \psi'$ acts on pure tensors via $(\psi \otimes \psi')(x \otimes x') = \psi(x) \otimes \psi'(x')$, and the action is unitary. By Lemma \ref{unitary-equivalence}, two equivalent unitary representations are in fact \emph{unitarily} equivalent, which means that there is some unitary isomorphism $\iota : V_1 \rightarrow X \otimes X'$ such that $\rho(g)v = \iota^{-1} (\psi \otimes \psi')(g)(\iota v)$ for all $v \in V_1$ and $g \in G$. 

We are finally ready to define our measure $\mu$. At this point the argument differs somewhat from that in the published version of the paper, which was incorrect. We thank Ashwin Sah, Mehtaab Sawhney and Yufei Zhao for drawing out attention to this issue and suggesting a correction, which we have gladly incorporated here.

Consider the map $\theta : S(X) \times S(X') \rightarrow V_1 \subset \C^d$ defined by $\theta(y,y') := \iota^{-1} (y \otimes y')$, and take $\mu$ to be the pushforward measure $\theta_*(\mu^*_{\Gamma_n} \times \delta_{y'})$ where $\mu^*_{\Gamma_n}$ is the probability measure on the unit sphere of $\C^n \cap \mathbf{1}^{\perp}$ defined in Proposition \ref{permprop-2} and $\delta_{y'}$ is the delta-measure localised at some arbitrarily chosen unit vector $y' \in V'$.

Let $v \in S(\C^d)$, and write $v = \sum_i v_i$ with $v_i \in V_i$. The $V_i$ are orthogonal and each is preserved by $\rho(g)$, and therefore the $\rho(g) v_i$ ($i \neq 1$) are orthogonal to $V_1$. Since $\mu$ is supported on $V_1$, it follows the preceding observation, the change of variables formula and the fact that $\iota$ is unitary that
\begin{align}\nonumber
\int & \sup_{g \in G} |\langle \rho(g) v,  w\rangle |^2 d \mu(w) \\ \nonumber & = \int \sup_{g \in G} |\langle \rho(g) v_1, w \rangle|^2 d\mu(w)  \\ \nonumber & = \iint \sup_{g \in G} |\langle \rho(g) v_1, \theta(y,y')\rangle |^2 d\mu^*_{\Gamma_n}(y) \\ \nonumber & = \iint \sup_{g \in G} | \langle \iota^{-1} (\psi \otimes \psi')(g)(\iota v_1), \iota^{-1} (y \otimes y')\rangle |^2 d\mu^*_{\Gamma_n}(y)  \\ & = \iint \sup_{g \in G} | \langle (\psi \otimes \psi')(g)(\iota v_1),  y \otimes y'\rangle |^2 d\mu^*_{\Gamma_n}(y) .\label{star-form}
\end{align}
Note that $\psi$ extends to a (reducible) unitary action on all of $\C^n$, rather than just the subspace $X = \mathbf{1}^{\perp}$, this being given by coordinate permutations, and similarly $\psi \otimes \psi'$ extends to an action on $\C^n \times X'$ . It is convenient to abuse notation and write $\psi$, $\psi \otimes \psi'$ for these extended actions as well. To bound \eqref{star-form}, it is then convenient to write
\[ \iota v_1 = \sum_{j = 1}^n e_j \otimes x'_j,\] where $e_1,\dots, e_n$ are the standard basis of $\C^n$. (Additionally, since $\iota v_1$ lies in $X \otimes X'$, we have $\sum_j x'_j = 0$, but we will not use this fact.)

Note for future reference that
\begin{equation}\label{unit} 1 \geq \Vert \iota v_1 \Vert^2 = \sum_{i,j} \langle e_i \otimes x'_i\rangle \langle e_j \otimes x'_j\rangle = \sum_i \Vert x'_i \Vert^2.\end{equation}
Denote by $I$ the expression in \eqref{star-form} which we seek to bound above. Then we have
\begin{align*}
I = & \int \sup_{\substack{a \in A_n \\ h \in H}} \big|\sum_{j=1}^n \langle (\psi \otimes \psi')(a \times h)(e_j \otimes x'_j), y \otimes y' \rangle \big|^2 d\mu^*_{\Gamma_n}(y)  \\ 
& =  \int \sup_{\substack{a \in A_n \\ h \in H}} \big| \langle \psi(a) e_j, y\rangle \langle \psi'(h) x'_j, y'\rangle   \big|^2 d\mu^*_{\Gamma_n}(y) \\
& = \int \sup_{\substack{a \in A_n \\ h \in H}} \big| \sum_{j=1}^n y_j \langle \psi'(h) x'_{\pi_a^{-1}(j)}, y'\rangle \big|^2 d\mu^*_{\Gamma_n}(y),
\end{align*}
where here $\pi_a$ is the permutation defined by $\psi(a) e_j = e_{\pi_a(j)}$, $y_j$ denotes the $j$th coordinate of $y_j \in S(\C^n)$ (and, in the sum over $j$, we changed variables, replacing $j$ by $\pi_a^{-1}(j)$). 
Therefore
\begin{align*} I & \leq \int \sup_{\substack{\pi \in S_n \\ h \in H}} \big| \sum_{j=1}^n \langle \psi'(h) x'_{\pi(j)}, y'\rangle y_j \big|^2 d\mu^*_{\Gamma_n}(y) \\ & \leq  \int \sup_{\pi \in S_n} \big| \sum_{j=1}^n \sup_{h \in H} | \langle \psi'(h) x'_{\pi(j)}, y'\rangle | |y_j|  \big|^2 d\mu^*_{\Gamma_n}(y).\end{align*}
By the key property of $\mu^*_{\Gamma_n}$ (see Proposition \ref{permprop-2}, and take $v_i := \sup_{h \in H} |\langle \psi'(h) x'_i, y' \rangle |$ there), it follows that 
\[ I \ll \frac{1}{\log n}\sum_{j = 1}^n \sup_{h \in H} |\langle \psi'(h) x'_j, y'\rangle |^2 .\]
Since $\psi'$ acts unitarily and $\Vert y' \Vert = 1$, we have $\sup_{h \in H} |\langle \psi'(h) x'_i, y'\rangle | \leq \Vert x'_i \Vert$, and so it follows from \eqref{unit} that
\[ I \ll \frac{1}{\log n}\sum_{i = 1}^n \Vert x'_i \Vert^2 \leq \frac{1}{\log n}.\]
Recall that $I$ was defined to be the expression in \eqref{star-form}. Thus we have shown that 
\[ \int \sup_{g \in G} |\langle \rho(g) v, w \rangle |^2 d \mu(w) \ll \frac{1}{\log n}.\]
Recalling that $n \geq d/\log^4 d$, so $\log n \gg \log d$, this completes the proof of \eqref{claim-1}. \end{proof}

\section{A consequence of Clifford theory}

The only use we will make of primitivity is through the following result, which is part of Clifford Theory.

\begin{lemma}\label{lem51}
Suppose that $G \leq U_d(\C)$ is primitive, and let $N \lhd G$ be a normal subgroup. Then 
\begin{enumerate}
\item $\C^d$ is a direct sum of isomorphic irreducible $N$-representations;
\item If $N$ is abelian, then it is contained in $Z_d = \{ \lambda I_d : |\lambda| = 1\}$.
\end{enumerate} 
\end{lemma}
\begin{proof}
(1) See, for example, \cite[Corollary 6.12]{isaacs}. Part (2) could be deduced from (1), but a short and direct proof using linear algebra is possible. We may decompose $\C^d$ as an orthogonal direct sum of simultaneous eigenspaces for the $n \in N$. That is, we have $\C^d = \bigoplus_{i = 1}^k V_i$ where for $n \in N$ and $v \in V_i$ we have $n v = \phi_i(n) v$, with $\phi_1,\dots, \phi_k \in \hat{N}$ being distinct characters on $N$. If $V_1 = \C^d$ then $N \subset Z_d$ and we are done. Otherwise, if $g \in G$ then for any $n \in N$ and $v \in V_i$ we have $ng  v = g(g^{-1}ng) v = g \phi_i(g^{-1} n g) v$, and so $gV_i = V_{\kappa_g(i)}$, where $\kappa_g(i)$ is the value of $j$ (which must exist) such that $\phi_{j}(n) := \phi_i(g^{-1} n g)$. It follows that $g$ permutes the $V_i$, which are therefore a system of imprimitivity for $G$. This is contrary to hypothesis.
\end{proof}

\emph{Remark.} The proof of (2) using linear algebra may in fact be generalised to give a proof of (1); the basic point is that the \emph{isotypic components} of $\C^d$ as an $N$-representation form a system of imprimitivity for $G$, and so there must only be one of them. Moreover, the different isotypic components are orthogonal (see \cite[Lemma 3.4.21]{kowalski}) and so the system of imprimitivity produced is automatically seen to be orthogonal, in theory\footnote{However, this would be a somewhat unsatisfactory way to proceed in that we would need, throughout the paper, to work with a notion of ``$o$-primitive'', by which we would mean a subgroup of $\U^d(\C)$ with no orthogonal system of imprimitivity: Lemma \ref{orth} of course shows that this coincides with the usual notion of imprimitivity.} bypassing the need for Lemma \ref{orth}.

We will use the following corollary of Lemma \ref{lem51} several times.

\begin{corollary}\label{prim-centres}
Let $G \leq \U(\C^d)$ be primitive. Then the centre of any normal subgroup of $G$ is contained in $Z_d$.
\end{corollary}
\begin{proof} 
Let $H \lhd G$ be normal. The center $Z(H)$ is characteristic in $H$, and hence $Z(H) \lhd G$. The claim now follows from Lemma \ref{lem51} (2).\end{proof}

\section{Using the generalised Fitting subgroup $F^*(G)$}

Our aim in this somewhat lengthy section is to prove Proposition \ref{inverse-collins}. We follow Collins \cite{collins1}, but since our aims are different we must restructure his argument. However, since our interest lies in the asymptotic regime (large groups), some minor simplifications to the argument are possible.

The key idea\footnote{Or at least what appears to me to be the key idea.} of Collins' paper is that it is useful to control $G$ using the \emph{generalised Fitting subgroup} $F^*(G)$. We give the definition below. There are two ways in which $F^*(G)$ controls $G$. Firstly, the index of $F^*(G)$ in $G$ is relatively small: the key statement here is Lemma \ref{lem5.5}. Secondly, in the primitive case the representation theory of $F^*(G)$ exerts control over that of $G$, as we show in Corollary \ref{lem5.6}.

\subsection{$F^*(G)$: definition and basic facts}

Recall that a finite group $Q$ is said to be \emph{quasisimple} if it is perfect (i.e. has no abelian quotients, or equivalently $[Q,Q] = Q$) and if $Q/Z(Q)$ is simple. If $G$ is a finite group, a subgroup $H \leq G$ is said to be \emph{subnormal} if there is a chain $H = H_k \lhd H_{k-1} \lhd \dots \lhd H_1 = G$.

We now give the definition of the generalised Fitting subgroup. The definition itself is not especially important to us. The properties we will need are summarised in Proposition \ref{fitting-props} below, which is of far greater consequence.

\begin{definition}
Then $F^*(G)$ is defined to be the product $\Gamma_1 \cdots \Gamma_m$ of the following subgroups of $G$:
\begin{itemize}
\item the $p$-cores $O_p(G)$, $p$ a prime, these being the maximal normal $p$-subgroups of $G$;
\item for all quasisimple groups $Q$, the products $\prod_{G_i \cong Q} G_i$, where the $G_i$ are the \emph{components} of $G$, that is to say the subnormal quasisimple subgroups of $G$. 
\end{itemize}
By convention we take all the $\Gamma_i$ to be nontrivial (thus, for example, if $G$ does not have a component isomorphic to some $Q$ then we do not include the empty product $\prod_{G_i \cong Q} G_i$ among the $\Gamma_i$).
\end{definition}

\emph{Remarks.} All of the products mentioned in the definition turn out to be central, so the orders of them are immaterial. For references, see the proof of the following proposition.

It is convenient to call the $\Gamma_i$ the\footnote{We caution that this is not standard terminology, and indeed our presentation in this section diverges a little from Collins' presentation in our explicit bundling together of isomorphic components $G_i$. } ``Fitting components'' of $G$.

\begin{proposition}\label{fitting-props}
Let $G$ be a finite group, and $F^*(G) = \Gamma_1 \dots \Gamma_m$ its generalised Fitting subgroup as defined above. Then
\begin{enumerate}
\item Each $\Gamma_i$ is well-defined (in the sense that the components $G_i$ commute with one another, so the order of the product is immaterial);
\item $\Gamma_i$ is normal in $G$;
\item The $\Gamma_i$ commute with one another, that is to say the product $\Gamma_1 \cdots \Gamma_m$ is central;
\item If $g \in G$ commutes with everything in $F^*(G)$, then $g \in F^*(G)$.
\end{enumerate}
\end{proposition}
\begin{proof} 
For (1), see \cite[31.5]{aschbacher}. For (2), note first that this is simply a part of the definition when $\Gamma_i$ is one of the $p$-cores. To see that $\Gamma_i = \prod_{G_i \cong Q} G_i$ is normal in $G$, note that conjugation must take any component $G_i$ to another component $G_j$, isomorphic to $G_i$, since conjugation preserves the property of being subnormal. For (3), it is a well-known fact (see Lemma \ref{p-core-commute} for references) that the $p$-cores $O_p(G)$ commute with one another. Moreover the $p$-cores commute with the components by \cite[31.4]{aschbacher} (strictly speaking, this tells us that if $G_i$ does not commute with $O_p(G)$ then we would have to have $G_i \subset O_p(G)$, and so $G_i$ is a $p$-group; but every $p$-group is nilpotent and thus certainly not quasisimple).
Finally, we have already remarked that the components commute with one another. These facts together imply (3). Finally, part (4) is \cite[31.13]{aschbacher}, where it is stated in the more succinct form $C_G(F^*(G)) \subset F^*(G)$.
\end{proof}

\subsection{The primitive case and extra-special groups}\label{sec62}

From now on we specialise to the case $G \leq \U(\C^d)$ primitive. In this case, as noted by Collins, we may say more about the structure of the $p$-cores. The reader may wish to refer to Aschbacher \cite{aschbacher-paper} (Collins contains much the same material, but we had some trouble filling in the details in characteristic 2, a task which \cite{aschbacher-paper} offers some help with). The relevant part of \cite{aschbacher-paper} is \S 1.7. $G$ satisfies the hypotheses there, namely that $Z(G)$ is the largest abelian normal subgroup of $G$ and is cyclic, by Lemma \ref{lem51} (2).  The conclusion is that if $\Gamma = O_p(G)$ is a $p$-core then $\Gamma = Z(\Gamma) E$, where $E$ is either trivial or \emph{extraspecial}. A group $E$ is called \emph{extra-special} if it is a $p$-group, if its centre $Z(E)$ is cyclic of order $p$, and if the quotient $E/Z(E)$ is isomorphic to $(\Z/p\Z)^m$ for some $m > 0$. For more facts about extra-special groups, see Proposition \ref{extra-special-prop}.  Since $Z(\Gamma)$ is characteristic in $\Gamma$, it is normal in $G$ and hence, by Lemma 5.1 (2), is contained in $Z_d$. Thus $\Gamma$ is a central product $ZE$ with $Z$ cyclic and $E$ either trivial or extraspecial. 

It follows that in the case $G$ primitive the possible Fitting components of $G$ belong to a somewhat restricted class, defined as follows.

\begin{definition}\label{f-prim-def}
Write $\mathscr{F}$ for the class of all finite groups $\Gamma$ which are either (i) a $p$-group which is a central product $Z E$, with $Z$ cyclic and $E$ is either trivial or extra-special or (ii) a central product of groups, each isomorphic to some quasisimple group $Q$.
\end{definition}
To reiterate, every Fitting component of a primitive group $G \leq \U(\C^d)$ lies in $\mathscr{F}$.

\subsection{Bounding the index $[G : F^*(G)]$}

Suppose that $\Gamma$ is a Fitting component of some primitive group $G \leq \U(\C^d)$. Since $\Gamma \lhd G$, $G$ acts on $\Gamma$ by conjugation. If $\Gamma = ZE$ is of type (i) in Definition \ref{f-prim-def} then, by the discussion at the start of subsection \ref{sec62}, the  action of $G$ on $Z$ (which is contained in $Z_d = \{ \lambda I_d : |\lambda| = 1\}$) is trivial. If $\Gamma$ is of type (ii) in Definition \ref{f-prim-def}, then this conjugation action permutes the (mutually isomorphic) quasisimple groups comprising $\Gamma$. These observations motivate the following definitions.

\begin{definition}\label{aut-prime}
Suppose that $\Gamma \in \mathscr{F}$. Write $\Aut'(\Gamma)$ for the subgroup of $\Aut(\Gamma)$ which, in case (i) of Definition \ref{f-prim-def}, act trivially on $Z$, and which in case (ii) permute the (mutually isomorphic) quasisimple groups comprising $\Gamma$. Write $\Out'(\Gamma)$ for the quotient of $\Aut'(\Gamma)$ by the normal subgroup of inner automorphisms of $\Gamma$.
\end{definition}

\begin{lemma}\label{lem5.5}
Suppose that $G \leq \U(\C^d)$ is primitive, and that $F^*(G) = \Gamma_1 \cdots \Gamma_m$ is its generalised Fitting subgroup. Then we have
\[ [G : F^*(G)] \leq \prod_{i = 1}^m |\Out'(\Gamma_i)|.\] 
\end{lemma}
\begin{proof}
By the remarks at the start of the subsection  we have a homomorphism 
\[ \pi : G \rightarrow \times_{i = 1}^m \Aut'(\Gamma_i)\] induced by conjugation. Composing with the quotient map gives a homomorphism
\[ \tilde\pi : G \rightarrow \times_{i = 1}^m \Out'(\Gamma_i).\] 
We claim that
\begin{equation}\label{keyclaim} \ker \tilde\pi \subset F^*(G), \end{equation} which immediately implies the lemma.
Suppose that $g \in \ker \tilde\pi$. This means that for each $i$ there is some $\gamma_{i,g} \in \Gamma_i$ such that $g^{-1} x g = \gamma_{i,g}^{-1} x \gamma_{i,g}$ for all $x \in \Gamma_i$, or in other words $g \gamma_{i,g}^{-1}$ commutes with everything in $\Gamma_i$. Since the $\Gamma_i$ commute with one another, the product $\tilde g := g (\prod_{i = 1}^m \gamma_{i,g}^{-1})$ is well-defined (does not depend on the ordering) and it commutes with all of the $\Gamma_i$. For example if $x \in \Gamma_2$ then
\[ \tilde g x = (g\gamma_{2,g}^{-1}) (\prod_{i \neq 2} \gamma_{g,i}^{-1}) x = (g\gamma_{2,g}^{-1}) x \prod_{i \neq 2} \gamma_{g,i}^{-1} = x(g\gamma_{2,g}^{-1})   \prod_{i \neq 2} \gamma_{g,i}^{-1} = x \tilde g.\] It follows that $\tilde g$ commutes with everything in $F^*(G)$. By Proposition \ref{fitting-props} (4) this implies that $\tilde g \in F^*(G)$, from which it follows immediately that $g \in F^*(G)$. This establishes the claim \eqref{keyclaim}.
\end{proof}

\subsection{Representations of $G$ and $F^*(G)$}

If $\Gamma$ is a finite group then we write $D(\Gamma)$ for the smallest dimension of a faithful irreducible representation of $\Gamma$. The next lemma applies to any groups $\Gamma_i$ (they need not be Fitting components).

\begin{lemma}\label{lem5.5-reps}
Suppose that $\Gamma_1\cdots\Gamma_m$ is a central product of finite groups. Then $D(\Gamma_1 \cdots \Gamma_m) \geq \prod_{i = 1}^m D(\Gamma_i)$.
\end{lemma}
\begin{proof}
Let $\rho : \Gamma_1 \cdots \Gamma_m \rightarrow \U(V)$ be a faithful irreducible representation of $\Gamma_1 \cdots \Gamma_m$ of minimal degree. Since the product is central, there is a surjective homomorphism $\pi : \times_{i = 1}^m \Gamma_i \rightarrow \Gamma_1 \cdots \Gamma_m$ (where, to be clear, $\times_{i = 1}^m \Gamma_i$ is the ``external'' direct product of the $\Gamma_i$, thought of as abstract groups). Thus we get an irreducible representation $\tilde \rho : \times_{i = 1}^m \Gamma_i \rightarrow \U(V)$, defined by $\tilde\rho = \rho \circ \pi$.  By Proposition \ref{product-rep}, $\tilde\rho$ is equivalent to $\tilde\tau_1 \otimes \cdots \otimes \tilde\tau_m$, where $\tilde\tau_j$ is the pullback of an irreducible representation $\tau_j : \Gamma_j\rightarrow \U(V_i)$ under the natural projection $\pi_j : \times_{i = 1}^m \Gamma_i \rightarrow \Gamma_j$. Noting that $\ker(\tilde\rho) \cap \Gamma_i$ is trivial for all $i$ (since $\rho$ is a faithful representation of $\Gamma_1 \cdots \Gamma_m$, and hence certainly of $\Gamma_i$), we see that each $\tau_i$ is faithful. It follows that 
\[ D(\Gamma_1 \cdots \Gamma_m) = \dim \tilde\rho = \prod_{i = 1}^m \dim \tilde\tau_i = \prod_{i = 1}^m \dim \tau_i \geq \prod_{i = 1}^m D(\Gamma_i).\]This completes the proof.
\end{proof}

\begin{corollary}\label{lem5.6}
Suppose that $G \leq \U(\C^d)$ is primitive, and that $\Gamma_1,\dots,\Gamma_m$ are its Fitting components. Then we have $\prod_{i = 1}^m D(\Gamma_i) \leq d$.
\end{corollary}
\begin{proof}
Consider $\C^d$ as a (faithful) representation of $F^*(G)$. Since $F^*(G)$ is a normal subgroup of $G$, it follows from Lemma \ref{lem51} that $\C^d$ is a direct sum of isomorphic irreducible $F^*(G)$-representations, say $\C^d = \bigoplus_{i = 1}^k V_i$. If (the action of $F^*(G)$ on) $V_1$ is not faithful then there are distinct $g, g' \in F^*(G)$ with $g v_1 = g' v_1$ for all $v_1 \in V_1$. Since the spaces $V_i$ are isomorphic as representations, we also have $g v_i = g' v_i$ for all $v_i \in V_i$ and so $g v = g' v$ for all $v \in \C^d$, contrary to the fact that $F^*(G)$ acts faithfully on $\C^d$. Thus $V_1$ is a faithful, irreducible representation of $F^*(G) = \Gamma_1 \cdots \Gamma_m$, and the result therefore follows from Lemma \ref{lem5.5-reps}.
\end{proof}

\subsection{Bounds for the Fitting components of primitive groups}

Recall the definition of the class $\mathscr{F}$ (Definition \ref{f-prim-def}) and of the space $\Aut'(\Gamma)$ (Definition \ref{aut-prime}). The main result of this subsection is the following somewhat technical bound. It is here that we use the CFSG and related results.
\begin{proposition}\label{prop68}
There is an absolute constant $C$ with the following property. Let $\Gamma \in \mathscr{F}$. Suppose that $\Gamma$ is not isomorphic to any of the alternating groups $A_n$, $n \geq 4$. Then 
\[ |\Gamma/Z(\Gamma)| |\Out' (\Gamma)| \ll e^{CD(\Gamma)^{2/3}}.\]
\end{proposition}
\begin{proof}
We consider five cases, as follows.
\begin{enumerate}
\item (Case 1a) (case (i) of Definition \ref{f-prim-def}). $\Gamma = ZE$ with $Z$ a cyclic $p$-group and $E$ trivial.
\item  (Case 1b) (case (i) of Definition \ref{f-prim-def}). $\Gamma$ is a $p$-group and $\Gamma$ is a central product $Z E$, where $Z$ is cyclic and $E$ is either trivial or extraspecial. 
\item (Case 2) $\Gamma$ is a central product of $r \geq 1$ copies of some quasisimple group $Q$ for which $Q/Z(Q) \neq A_n$;
\item (Case 3) $\Gamma$ is a central product of $r \geq 2$ copies of some quasisimple group $Q$ for which $Q/Z(Q) = A_n$;
\item (Case 4) $\Gamma/Z(\Gamma) = A_n$, but $\Gamma \neq A_n$.
\end{enumerate}
It is clear that this covers all cases. 

\emph{Case 1a.} This is somewhat trivial. We have $|\Gamma/Z(\Gamma) | = |\Aut'(\Gamma) | = 1$, $D(\Gamma) = 1$. 

\emph{Case 1b.} By Proposition \ref{extra-special-prop} (1) we have $|E| = p^{1 + 2m}$ for some $m \geq 1$. Suppose that $Z$ is cyclic of order $p^n$. By Proposition \ref{extra-special-prop} (2), we may pick generators $e_1,\dots, e_{2m}$ for $E$. Let $z$ be a generator for $Z$. Suppose that $\phi \in \Aut'(\Gamma)$. Then, since $\phi$ fixes $z$, it is completely determined by the images $\phi(e_i)$. Write $\phi(e_i) = z^{r_i} x_i$, where $r_i \in \Z/p^n \Z$ and $x_i \in E$. By Proposition \ref{extra-special-prop} (3), $e_i^{p^2} = x_i^{p^2} = 1$ for every $i$. Therefore
\[ z^{r_i p^2}  = (z^{r_i} x_i)^{p^2} = \phi(e_i)^{p^2} = \phi(e_i^{p^2}) = 1,\] and therefore $p^2 r_i = 0 \pmod{p^n}$. This means that there are only $p^2$ choices for $r_i$, for each $i$, and so the number of choices for $\phi(e_i)$ is at most $p^{2m + 3}$. We therefore have the bound 
\[ |\Out'(\Gamma)| \leq |\Aut'(\Gamma)| \leq p^{2m(2m+3)}.\]
Since $|\Gamma/Z(\Gamma)| \leq |E|$, it follows that 
\begin{equation}\label{upper-gam} |\Gamma/Z(\Gamma)| |\Out'(\Gamma)| \leq p^{1 + 2m} |\Out'(\Gamma)| \leq p^{4m^2 + 8m+1}.\end{equation}
On the other hand, Proposition \ref{extra-special-prop} (4) implies that $D(\Gamma) \geq D(E) \geq p^m$, and so 
\[ e^{CD(\Gamma)^{2/3}} \geq e^{Cp^{2m/3}}.\] It is clear that, for some choice of $C$, this is greater than the right-hand-side of \eqref{upper-gam} uniformly in $p, m$ (by a vast margin if either $p$ or $m$ is large). 

\emph{Case 2.}  This is where we use CFSG.  Proposition \ref{a5} (whose proof required CFSG) tells us that $D(Q) \geq e^{c\sqrt{\log Q}}$ for some constant $c > 0$. 
In fact, the much weaker inequality
 \begin{equation}\label{weak-qs-rep} D(Q) \geq \max(c_1 \log^3 |Q|,2)\end{equation} is sufficient for our purposes, but no result of this strength is known without the CFSG. By Lemma \ref{lem5.5-reps}, it follows that \begin{equation}\label{tensor} D(\Gamma) \geq D(Q)^r \geq \max(c^r_1 \log^{3r}|Q|, 2^r).\end{equation} By Lemma \ref{aut-lem} (which is also extremely crude, though this time elementary) we have
 \[ |\Gamma/Z(\Gamma)| |\Out'(\Gamma)| \leq |Q|^r e^{c_2r^2\log^2 |Q|} \leq e^{c_3 r^2 \log^2 |Q|}.\] Comparing with \eqref{tensor}, this is indeed bounded above by $e^{C D(\Gamma)^{2/3}}$ for some absolute constant $C$.

\emph{Case 3.} By choosing $C$ large enough, we may assume that $n \geq 15$. (This is not quite trivial, and requires the fact that for each $n$ there are only finitely many quasisimple $Q$ with $Q/Z(Q) \cong A_n$, but this follows from the finiteness of the Schur multiplier $H_2(A_n, \C^{\times})$.)  When $n \geq 15$, we can apply Proposition \ref{alt-prop} (4). This tells us that either $Q = A_n$ or $Q = \hat{A}_n$, the double cover of $A_n$. In the latter case it follows from Proposition \ref{alt-prop} (5) that $D(Q) \gg \log^3 |Q|$ (by a vast margin), and so we can proceed as in Case 2. 

Suppose, then, that $Q = A_n$. Thus, by Lemma \ref{lem5.5-reps} and Proposition \ref{alt-prop} (2), 
\[ D(\Gamma) \geq D(A_n)^r = (n-1)^r.\]
 Since $\Out'(\Gamma)$ is, by definition, only concerned only with those automorphisms which permute the $r$ factors of $A_n$, it is easy to see that 
\[ |\Out'(\Gamma)| \leq r! |\Out(A_n)|^r = r! 2^r, \] 
where here we used Proposition \ref{alt-prop} (3). Hence
\[ |\Gamma/Z(\Gamma)| |\Out'(\Gamma)| \leq (n!)^r r! 2^r.\]  If $C$ is large enough then this is indeed at most $e^{C(n-1)^{2r/3}}$, uniformly for all $r \geq 2$ and $n \geq 15$.

\emph{Case 4.} Once again we may assume that $n \geq 15$. By Proposition \ref{alt-prop} (4), $\Gamma$ is the double cover $\hat{A}_n$. Hence, by Proposition \ref{alt-prop} (5), $D(\Gamma) = 2^{\lfloor (n-2)/2\rfloor}$. Using the crude bound of Lemma \ref{aut-lem} to bound $|\Out'(\Gamma)|$, we see that the claimed inequality is true by a very large margin, the right-hand side being doubly-exponential in $n$ and the left-hand side being at most $e^{O(n^2 \log^2 n)}$.

\end{proof}

\subsection{Proof of Proposition \ref{inverse-collins}.}

\begin{proof}[Proof of Proposition \ref{inverse-collins}]
Suppose, as in the hypotheses of the proposition we are proving, that $G \leq \U_d(\C)$ is primitive and that $[G : Z_d \cap G] \geq d/\log d$. Let $F^*(G) = \Gamma_0 \Gamma_1 \cdots \Gamma_m$ be the generalised Fitting subgroup of $G$, where here $\Gamma_0$ is the product of all the abelian Fitting components of $G$ (if there are any), $\Gamma_1,\dots, \Gamma_{m'}$ are (nonabelian) alternating groups $A_{n_i}$ with $n_1 \geq \dots \geq n_m  \geq 4$, whilst the other $\Gamma_i$ are nonabelian quasisimple groups, none of them isomorphic to alternating groups.  Corollary \ref{lem5.6} tells us that  \begin{equation}\label{eqcor56ded}\prod_{i = 1}^m D(\Gamma_i) \leq d.\end{equation} We will use two (somewhat crude) consequences of this. First, that 
\begin{equation}\label{dD} \prod_{i > m'} D(\Gamma_i) \leq d  \end{equation}
(this is immediate) and second that 
\begin{equation}\label{alt-bd}  n_1 \ll 3^{-m'} d. \end{equation}
To see \eqref{alt-bd}, note first that the bound $D(A_{n_1}) \ll 3^{-m'} d$ follows immediately from \eqref{eqcor56ded} and the fact that $D(A_n) \geq 3$ whenever $n \geq 4$. To deduce \eqref{alt-bd}, we observe that $D(A_n) \geq \frac{3}{5} n$ for $n \geq 4$, the worst case being $n = 5$ where $A_5$ has a faithful representation of dimension $3$.

By Lemma \ref{lem5.5} we have
\[ [G : \Gamma_1 \cdots \Gamma_m] = [G : F^*(G)]   \leq \prod_{i = 1}^m |\Out'(\Gamma_i)|.\] By Corollary \ref{prim-centres}, $Z(\Gamma_i) = \Gamma_i \cap Z_d$ and $Z(G) = G \cap Z_d$. Note also that, by Proposition \ref{fitting-props} (4) (or straight from the definition), $Z(G) \subset F^*(G)$, and therefore $Z(G) = F^*(G) \cap Z_d$. It follows that 
\begin{equation}\label{gzg} [G : Z_d \cap G] = [ G : Z(G)] = [ G : F^*(G)] [ F^*(G) : F^*(G) \cap Z_d]. \end{equation}
Since the $\Gamma_i$ commute, there is a well-defined homomorphism
\begin{align*} \psi : \times_{i = 1}^m \Gamma_i/(\Gamma_i \cap Z_d)&  \rightarrow \Gamma_1 \cdots \Gamma_m/((\Gamma_1 \cdots \Gamma_m) \cap Z_d) \\ & = F^*(G)/(F^*(G) \cap Z_d),\end{align*} and therefore
\[ [ F^*(G) : F^*(G) \cap Z_d] \leq \prod_{i = 1}^m [\Gamma_i : \Gamma_i \cap Z_d] = \prod_{i = 1}^m |\Gamma_i/Z(\Gamma_i)|.\]
Comparing with \eqref{gzg} yields
\[ [G : Z_d \cap G] \leq \prod_{i = 1}^m |\Gamma_i/Z(\Gamma_i)| |\Out' (\Gamma_i)|.\] We must now estimate this.
For the product over $i > m'$, we use Proposition \ref{prop68}, whilst for $i \leq m$ we have $|\Out'(\Gamma_i)| \leq |\Out(A_{n_i})| \leq 4$ by Proposition \ref{alt-prop}. Recalling our assumption $[G : Z_d \cap G] \geq e^{d/\log d}$, this gives
\[ e^{d/\log d} \leq [G : Z_d \cap G] \leq \prod_{i = 1}^{m'} 2n_i! \prod_{i > m'} e^{CD(\Gamma_i)^{2/3}},\] and thus
\begin{equation}\label{to-use-44} \frac{d}{\log d} \ll \sum_{i \leq m'} n_i \log n_i + \sum_{i > m'} D(\Gamma_i)^{2/3}.\end{equation}
Assume (relabeling if necessary) that $D(\Gamma_{m})$ is the largest of the $D(\Gamma_i)$, $i > m'$. Since the $\Gamma_i$, $i \geq 1$ are all nonabelian we have $D(\Gamma_i) \geq 2$ and therefore
\begin{align*}
\sum_{i > m'} D(\Gamma_i)^{2/3}  & \leq (m- m') D(\Gamma_m)^{2/3} \\ & \leq (m - m') 2^{-2/3(m - m' - 1)} \prod_{i > m'} D(\Gamma_i)^{2/3} \\ & \ll \prod_{i > m'} D(\Gamma_i)^{2/3} \leq d^{2/3},
\end{align*}
uniformly in $m,m'$; the last step follows from \eqref{dD}. It follows that if $d$ is large enough then we may improve \eqref{to-use-44} to
\begin{equation}\label{to-use-45} \frac{d}{\log d} \ll \sum_{i \leq m'} n_i \log n_i.\end{equation}
Applying \eqref{alt-bd}, we have
\begin{align*} \sum_{i = 1}^{m'} n_i \log n_i  & \leq m' n_1 \log n_1 \\ &  \ll m' 3^{-m'/2} n_1^{1/2} d^{1/2}\log n_1 \\ & \ll n_1^{1/2}d^{1/2} \log n_1,\end{align*} uniformly in $m'$.  Comparing with \eqref{to-use-45} we obtain $n_1 \gg d/\log^4 d$, as required. \end{proof}

This completes the proof of all the statements claimed in the paper.
\appendix

\section{Facts from group theory and representation theory}

\begin{proposition}\label{product-rep}
Let $G_1,\dots, G_k$ be finite groups. Then the irreducible complex representations of $G_1 \times \cdots \times G_k$ are precisely the tensor products $\rho_1 \otimes \cdots \otimes \rho_k$, where $\rho_i$ is the pullback of an irreducible representation of $G_i$ under the natural projection from $G_1 \times \cdots \times G_k$ to $G_i$.
\end{proposition}
\begin{proof} This is a standard fact from representation theory and may be shown using characters: see, for example, \cite[Theorem 19.18]{james-liebeck}.\end{proof}

\begin{lemma}\label{aut-lem}
Let $G$ be a group. Then $\log |\Aut(G)| \ll \log^2 |G|$.
\end{lemma}
\begin{proof}
Greedily pick a generating set $\{x_1,\dots, x_k\}$ for $G$ by taking $x_{j+1}$ to be any element not in $\langle x_1,\dots, x_j\rangle$, if there is one. Since $\langle x_1,\dots, x_j\rangle$ is a proper subgroup of $\langle x_1,\dots, x_{j+1}\rangle$, we have $|\langle x_1,\dots, x_j\rangle| \geq 2^j$, and therefore $k \leq \log_2 |G|$. But any automorphism $\phi \in \Aut(G)$ is completely determined by the $\phi(x_i)$.
\end{proof}

\begin{lemma}\label{unitary-equivalence}
Let $G$ be a finite group. Suppose that $G$ has unitary representations $\rho_i : G \rightarrow \U(V_i)$, $i = 1, 2$ and that these are equivalent in the usual sense of representation theory: thus there exists a linear isomorphism $\pi : V_1 \rightarrow V_2$  which is $G$-intertwining in the sense that $\pi \circ \rho_1(g) = \rho_2(g) \circ \pi$ for all $g \in G$. Then $\pi$ is a scalar multiple of a unitary map, and in particular the two representations are unitarily equivalent.
\end{lemma}
\begin{proof}
Define a linear automorphism $\pi : V_1 \rightarrow V_1$ via $\langle \psi x, y \rangle_{V_1} = \langle \pi x, \pi y\rangle_{V_2}$ for all $x,y \in V_1$. If $g \in G$ then we compute that \begin{align*} \langle \rho_1(g)\psi x,& \rho_1(g) y \rangle_{V_1} = \langle \psi x, y \rangle_{V_1} = \langle \pi x, \pi y \rangle_{V_2} \\ & = \langle \rho_2(g) \pi x, \rho_2(g) \pi y\rangle_{V_2} =\langle \pi \rho_1(g) x,\pi \rho_1(g) y\rangle_{V_2} \\ & =  \langle \psi \rho_1(g) x,\rho_1(g) y \rangle_{V_1}.\end{align*} From this it follows that $\rho_1(g) \circ \psi = \psi \circ \rho_1(g)$, or in other words that $\psi$ is $G$-intertwining. By Schur's lemma, $\psi$ is a scalar multiple of the identity, and hence $\pi$ is a scalar multiple of a unitary map.
\end{proof}

\emph{Remark.} The same is true without the assumption of irreducibility; see \cite{mathoverflow-unitary} for discussion.

\begin{proposition}\label{extra-special-prop}
Let $E$ be an extra-special $p$-group. We have the following statements.
\begin{enumerate}
\item $E$ has size $p^{1 + 2m}$ for some integer $m \geq 1$;
\item $E$ is generated by $2m$ elements;
\item the exponent of $E$ divides $p^2$; 
\item the dimension $D(E)$ of the smallest faithful representation of $E$ is at least $p^m$.
\end{enumerate}
\end{proposition}
\begin{proof} See \cite[\S 5.5]{gorenstein}.\end{proof}

\begin{proposition}\label{a5}
There is a constant $c > 0$ such that, if $Q$ is a quasisimple group which is neither abelian nor alternating, we have $D(Q) \geq e^{C \sqrt{\log |Q|}}$.
\end{proposition}
\begin{proof}
It follows from \cite[Chapter 33]{aschbacher} that if $Q$ is quasisimple with corresponding simple group $\tilde Q = Q/Z(Q)$ then $|Z(Q)|$ is bounded in size by the cardinality of the \emph{Schur multiplier} $|H^2(\tilde Q, \C^{\times})|$. The Schur multipliers of all finite simple groups are known (of course, this requires CFSG) and in fact they are universally bounded in size except for the groups $\mbox{PSL}_n(\F_q)$ and $\mbox{PSU}_n(\F_q)$, in which case they have size $\leq \max(n+1, q+1)$.  Thus in all cases $|Z(Q)| \ll \sqrt{\log |Q|}$, and so
\begin{equation}\label{qqtile}
|\tilde Q| \gg |Q|(\log |Q|)^{-1/2}.
\end{equation}
We note that vastly inferior bounds would suffice for our purposes, but what is available without the CFSG is ridiculously weak: see \cite{tao-mathoverflow} from some discussion of this point.

Now we refer to the table \cite[p 419]{landazuri-seitz}. This table gives lower bounds for the degrees of projective representations of Chevalley groups. For our purposes, it is important to know that (i) ``projective representations'' of $G$ are in 1-1 correspondence with representations of central extensions of $G$; that (ii) ``Chevalley groups'' means 
``groups of Lie type'' rather than the more restricted notion of Chevalley group one sometimes sees, and (iii) the CFSG states (in its ``rough'' form) that all sufficiently large nonabelian finite simple groups are either alternating groups, or groups of Lie type. 
An inspection of the table reveals that $D(Q) \geq e^{c (\log |\tilde Q|)^{1/2}}$ when $Q$ is quasisimple and not an alternating group. 
The statement of the proposition follows immediately from this and \eqref{qqtile}.  
 \end{proof}

 \begin{proposition}\label{alt-prop}
 Let $n \geq 15$. Then all of the following statements hold.
 \begin{enumerate}
 \item $Z(A_n)$ is trivial.
 \item The unique irreducible representation $\rho$ of $A_n$ with $1 < \dim \rho < \frac{1}{2}n(n-3)$ is the permutation representation on $\{ z \in \C^n : z_1 +\dots + z_n = 0\}$, which has dimension $n - 1$;
 \item $\Aut(A_n) \cong S_n$, and hence $|\Out(A_n)| = 2$;
 \item Apart from $A_n$ itself, there is a unique quasisimple group $Q = \hat{A}_n$, the ``double cover'' of $A_n$, such that $Q/Z(Q) \cong A_n$;
 \item $D(\hat{A}_n) = 2^{\lfloor (n - 2)/2\rfloor}$.
 \end{enumerate}
 Moreover $|\Out(A_n)| \leq 4$ for all $n \geq 4$.
 \end{proposition}
 \begin{proof}
 (1) and (3) are very well-known. Point (2) is essentially in Rasala \cite{rasala} (but see also \cite[Corollary 5]{tong-viet}).
 For (4), it is well-known that when $n \geq 8$, the Schur multiplier $H_2(A_n, \C^{\times})$ is cyclic of order $2$, and that the nontrivial element corresponds to the (Schur) double cover $\hat{A}_n$. Finally, (5) goes back to Schur, at least for large $n$, but see \cite[Main Theorem]{kt}. 
 \end{proof}
 
\begin{lemma}\label{p-core-commute}
Let $G$ be a finite group. Then any two $p$-cores $O_{\pi}(G)$, $O_{\pi'}(G)$, $\pi,\pi'$ distinct primes, commute. 
\end{lemma} 
\begin{proof}
This follows easily from the basic theory of finite nilpotent groups, for which \cite{hall} is a good reference. First it follows from the ``if'' direction of \cite[Theorem 10.3.4]{hall} that both $O_{\pi}(G)$ and $O_{\pi'}(G)$ are nilpotent normal subgroups of $G$. By \cite[Theorem 10.3.2]{hall}, the product $H := O_{\pi}(G) O_{\pi'}(G)$ is a nilpotent subgroup of $G$. Now it follows from the ``only if'' direction of \cite[Theorem 10.3.4]{hall} that $H$ is in fact a direct product of $O_{\pi}(G)$ and $O_{\pi'}(G)$.  
\end{proof}

\end{document}